\documentclass[12pt]{article}

\usepackage{amsthm}%
\usepackage{amsmath,amssymb}
\usepackage{graphicx}
\usepackage[colorlinks=true,citecolor=black,linkcolor=black,urlcolor=blue]{hyperref}
\usepackage{url}
\usepackage{cleveref}

\newtheorem{theorem}{Theorem}

\newtheorem{lemma}[theorem]{Lemma}

\def\paren#1{\left( #1 \right)}
\def\acc#1{\left\{ #1 \right\}}

\renewcommand{\le}{\leqslant}
\renewcommand{\ge}{\geqslant}

\title{Avoiding large squares in trees and planar graphs}

\author{
Daniel Gon\c{c}alves, Pascal Ochem, Matthieu Rosenfeld\\
\small LIRMM, CNRS\\
\small Universit\'e de Montpellier\\
\small France\\
}
\begin{document}

\maketitle
\setcounter{footnote}{0}
\begin{abstract}
The Thue number $\pi(G)$ of a graph $G$ is the minimum number of colors needed to color $G$ without creating a square on a path of $G$.
For a graph class $C$, $\pi(C)$ is the supremum of $\pi(G)$ over the graphs $G\in C$.
The Thue number has been investigated for famous minor-closed classes: $\pi(tree)=4$, $7\le\pi(outerplanar)\le12$, and $11\le\pi(planar)\le768$.
Following a suggestion of Grytczuk, we consider the generalized parameters $\pi_k(C)$ such that only squares of period at least $k$ must be avoided. Thus, $\pi(C)=\pi_1(C)$.
We show that $\pi_5(tree)=2$, $\pi_2(tree)=3$, and $\pi_k(planar)\ge11$ for every fixed $k$.
\end{abstract}

\section{Introduction}\label{sec:intro}
A coloring of a graph $G$ is non-repetitive if the sequence induced by the colors of any path of $G$ is not a square.
The Thue number $\pi(G)$ of $G$ is the minimum number of colors needed in a non-repetitive coloring $G$.
Recall that the period of a square $uu$ is $|u|$.

For a graph class $C$, $\pi(C)$ is the supremum of $\pi(G)$ over the graphs $G\in C$.
Let $tw_k$ denote the class of graphs with treewidth at most $k$.
\begin{theorem}\label{thm:nonr}{\ }
\begin{itemize}
 \item $\pi(path)=3$~\cite{Thue06}
 \item $\pi(tree)=4$~\cite{4ColTree}
 \item $7\le\pi(outerplanar)\le12$~\cite{bv:2007}
 \item $\pi(tw_k)\le4^k$~\cite{kp:2008}
 \item $11\le\pi(planar\cap tw_3)\le\pi(planar)\le768$~\cite{DFJW,DEJWW}
\end{itemize}
\end{theorem}
Other types of coloring, namely proper coloring and star coloring~\cite{ACKKR04},
are defined by forbidding only squares of period 1 and squares of period 1 and~2, respectively.
The corresponding chromatic numbers $\chi$ and $\chi_s$ thus satisfy
$\chi(C)\le\chi_s(C)\le\pi(C)$ for every graph class $C$.

This paper investigates another variation of non-repetitive coloring,
suggested by Grytczuk, such that only large enough squares are forbidden.
The parameter $\pi_k(G)$ is the minimum number of colors needed to
color $G$ such that no squares of period at least $k$ appears in $G$.
We similarly define $\pi_k(C)$ for a graph class $C$, so that
$\pi(C)=\pi_1(C)\ge\pi_2(C)\ge\pi_3(C)\ge\cdots$

The case of words, i.e. infinite paths, is already settled
\begin{theorem}\label{thm:path}{\ }
\begin{itemize}
 \item $\pi_k(path)=3$ for $1\le k\le2$~\cite{Thue06}
 \item $\pi_k(path)=2$ for $k\ge3$~\cite{EJS1974}
\end{itemize}
\end{theorem}

We settle the case of trees in Section~\ref{tree}.
\begin{theorem}\label{thm:tree}{\ }
\begin{itemize}
 \item $\pi_1(tree)=4$
 \item $\pi_k(tree)=3$ for $2\le k\le4$
 \item $\pi_k(tree)=2$ for $k\ge5$
\end{itemize}
\end{theorem}
We also obtain a lower bound for planar graphs in Section~\ref{planar}.
\begin{theorem}\label{thm:planar}
for every fixed $k$, $\pi_k(planar\cap tw_3)\ge11$.
\end{theorem}
This disproves a conjecture of Grytczuk~\cite{grytczuk} that $\pi_k(planar)=4$ for some~$k$.

\section{Trees}\label{tree}
Colorings of trees have been considered~\cite{OchVas12} that minimize
the critical exponent of repetitions.
To avoid large squares, and for the same reasons as in~\cite{OchVas12},
we can assume without loss of generality that our colored tree is rooted
and that all the vertices at the same distance to the root have the same color.
So we only need to describe the word $w$ lying on one branch of the tree.
We adopt the counter-intuitive convention that the reading direction of $w$ goes towards the root.
Then every factor $fs$ of $w$ with $|s|=1$ should be such that $fsf^R$ (where $f^R$ is the reverse of $f$)
avoids the forbidden large squares.
Let $w_3$ be any infinite $\paren{\tfrac74^+}$-free ternary word.

We obtain $w$ by taking the image of any $\paren{\tfrac74^+}$-free ternary word by the following morphisms.

We use the $12$-uniform morphism $g_2$ to prove $\pi_2(tree)\le3$ and the $21$-uniform morphism $g_5$ to prove $\pi_5(tree)\le2$:

\begin{minipage}[b]{0.4\linewidth}
$$\begin{array}{c}
 g_2(\texttt{0})=\texttt{011220012201}\\
 g_2(\texttt{1})=\texttt{122001120012}\\
 g_2(\texttt{2})=\texttt{200112201120}\\ 
\end{array}$$
\end{minipage}
\begin{minipage}[b]{0.5\linewidth}
$$\begin{array}{c}
 g_5(\texttt{0})=\texttt{001101110001010110010}\\
 g_5(\texttt{1})=\texttt{001101110001001110101}\\
 g_5(\texttt{2})=\texttt{001101110001001101010}\\ 
\end{array}$$
\end{minipage}

\medskip
\noindent
A word $u$ is \emph{$d$-directed} if for every factor $f$ of $u$ of length $d$, the word $f^R$ is not a factor of~$u$.
To prove that a word is $d$-directed, it suffices to check its factors of length~$d$.
A word is $(\beta^+,n)$-free if it contains no repetition with exponent strictly greater than $\beta$ and period at least~$n$.
To prove the $(\beta^+,n)$-freeness, we use the method described in~\cite{Ochem2004}. This way, we obtain the following.
\begin{itemize}
 \item $g_2(w_3)$ is $3$-directed and $\paren{\tfrac{19}{10}^+,2}$-free.
 \item $g_5(w_3)$ is $20$-directed and $\paren{\tfrac{83}{42}^+,5}$-free.
\end{itemize}

Consider $g_2(w_3)$. For contradiction, suppose that $g_2(w_3)$ contains a factor $fs$ with $|s|=1$ 
such that the word $fsf^R$ contains a square of period $p\ge20$.
Since $fsf^R$ is a palindrome, we can assume that the center of the square is on the left of $s$.
Since $fs$ is $\paren{\tfrac{19}{10}^+,2}$-free, $fs$ must contain (as a suffix) a prefix of this square of length at least $p+1$
and at most $\tfrac{19}{10}p$. So $sf^R$ must contain (as a prefix) a suffix $x$ of this square of length at most $p$
and at least $\tfrac{1}{10}p+1\ge3$. Because of the square, $x$ appears both in $fs$ and $sf^R$.
Notice that $(fs)^R=sf^R$, so that $fs$ contains both $x$ and $x^R$.
This is a contradiction since $|x|\ge3$ and $fs$ is $3$-directed.
Finally, an exhaustive computer check shows that the words $fsf^R$ contain no square of period $p$ with $2\le p\le19$.

The proof for $g_5(w_3)$ is similar.

\section{Planar graphs}\label{planar}
We start with helpful results.
\begin{lemma}\label{plus4}
$\pi(planar\cap tw_3)\ge\pi(outerplanar)+4$
\end{lemma}
\begin{proof}
Let $H$ be an outerplanar graph such that $\pi(H)=\pi(outerplanar)$.
We construct a planar graph $G$ as follows.
We start with a matching $a_1b_1,a_2b_2,\cdots,a_{1000}b_{1000}$.
For every vertex $x$ in this matching, we add a copy $H_x$ of $H$ and we make $x$ adjacent to every vertex of $H_x$.
Finally we add two adjacent vertices $c$ and $d$ and we make them adjacent to every vertex in the matching.

For contradiction, suppose that $\pi(G)\le\pi(H)+3$ and let $m$ be a nonrepetitive coloring of $G$ using the alphabet $A=\acc{1,2,\cdots,\pi(H)+3}$.
Since $\pi(H)+3\le15$, two edges of the matching get the same colors.
Without loss of generality, $m(a_1)=m(a_2)=1$ and $m(b_1)=m(b_2)=2$. We can also assume that $m(c)=3$ and $m(d)=4$.
Let $u$ be a vertex in $H_{a_1}$. Because of the edge $ua_i$, $m(u)\ne1$. Also $m(u)\ne3$, since otherwise the path $ua_1ca_2$ would create the square $3131$.
Symmetrically, $m(u)\ne4$. Thus $m(H_{a_1})\subset A\setminus\acc{1,3,4}$.
Since $|A\setminus\acc{1,3,4}|=\pi(H)$, there exists a vertex $u$ in $H_{a_1}$ such that $m(u)=2$.
By swapping the roles of colors $1$ and $2$ in the argument above, we deduce that there exists a vertex $v$ in $H_{b_1}$ such that $m(u)=1$.
Now the path $ua_1b_1v$ creates the square $2121$, which is a contradiction.
\end{proof}

\begin{lemma}\label{path}
Let $k$ be a fixed integer and let $P$ be a path. In every proper coloring of $P$ avoiding squares of period at least $k$,
every subpath of $P$ with $4k$ vertices contains at least $3$ colors.
\end{lemma}

\begin{proof}
A proper 2-coloring of the path of $4k$ vertices contains the square $(01)^{2k}$ of period~$2k$.
So at least $3$ colors are needed to avoid squares of period at least $k$.
\end{proof}

\begin{lemma}\label{lem:outer}
For every fixed $k$, there exists an outerplanar graph that admits no
proper 5-coloring avoiding squares of period at least $k$.
\end{lemma}

\begin{proof}
Our outerplanar graph $G$ has a root vertex and vertices at distance $i$ from the root are said to be on level $i$.
A vertex on level $i+1$ has exactly one neighbor on level $i$.
The neighborhood on level $i+1$ of vertex on level $i$ induces a very long path.
Finally, $G$ contains $10k$ levels.

For contradiction, suppose that $G$ has a proper 5-coloring avoiding squares of period at least $k$ using the colors $\acc{0,1,2,3,4}$.
Without loss of generality, the root (on level~$0$) is colored~$0$.
Without loss of generality, the very long path on level 1 contains two non-intersecting occurrences of a long factor of the form $w1$.
Now we consider the very long path on level 2 adjacent to the suffix letter~$1$ of the rightmost occurrence of $w1$.
It does not contain color $0$, since otherwise we would have the long square $w10w10$ such that the first $0$ is the root and the second $0$ is on level 2.
So it must be colored with the remaining colors $\acc{2,3,4}$.
By Lemma~\ref{path}, each of the three colors in $\acc{2,3,4}$ are recurrent in our very long path.
In particular, it contains two non-intersecting occurrences of a long factor of the form $z2$.
On level 3, below the suffix letter $2$ of the rightmost occurrence of $z2$, the very long path does not contain color $2$,
since otherwise we would have the long square $z21z21$.
So this very long path on level $3$ must be colored with letters $\acc{0,3,4}$.

Continuing this line of reasonning leads to the existence of a downward path from the root such that the vertex on level $i$ is colored $i\pmod{5}$.
So $G$ contains the long square $(01234)^{2k}$.
\end{proof}

We define the sequence $G_i$ of planar triangulations with treewidth $3$ such that
$G_0$ is $K_4$ and $G_{i+1}$ is obtained by adding a vertex of degree $3$ in every face of $G_i$.

let us give some properties of the sequence of $G_i$.
\begin{lemma}\label{proper}
If $xy$ is an edge of $G_i$, then $x$ and $y$ are adjacent to every
vertex of a path on~$t$ vertices in $G_{i+t}$.
Moreover, $G_i$ and this path have an empty intersection.
\end{lemma}

We define the sequence $U_i$ of outerplanar graphs as follows:
\begin{itemize}
 \item every graph $U_i$ has exactly one main edge.
 \item $U_0$ is $K_2$.
 \item $U_{i+1}$ is obtained from a copy of $U_i$ with main edge $ab$ and a copy of $U_i$ with main edge $cd$ by identifying the vertices $b$ and $c$
 and adding the edge $ad$. Then $ad$ is the main edge of $U_{i+1}$.
\end{itemize}
It is not hard to check that every outerplanar graph is the subgraph of some graph $U_i$.
\begin{lemma}\label{outer}
If $x$ is a vertex of $G_i$, then $x$ is adjacent to every vertex of a copy of $U_t$ in $G_{i+t+2}$.
Moreover, $G_i$ and this copy of $U_t$ have an empty intersection.
\end{lemma}

We are ready to prove Theorem~\ref{thm:planar}.
For contradiction, suppose that $q$ is the smallest integer such that all the $G_i$'s
can be colored with $10$ colors such that the period of every square is at most $q$.
By Lemma~\ref{proper}, we can assume that the $10$-coloring is proper and we use the alphabet $A_{10}=\acc{0,\cdots,9}$.

Recall the $\pi(outerplanar)\ge7$~\cite{bv:2007}.
Using Lemma~\ref{plus4}, we obtain $\pi(planar\cap tw_3)\ge11$.
Let $W$ be a witness of this lower bound.
Then every proper $10$-coloring of a graph $G_i$ that contains $W$
must contain a square with period at least 2. This shows that $q\ge2$.

By definiton of $q$, there exists $j$ such that every proper $10$-coloring of $G_j$ contains a square of period $q$.
By permuting the colors, we can assume that this square is $S=0\cdots10\cdots1$ where $|0\cdots1|=q$.
By Lemma~\ref{proper} and Lemma~\ref{path}, there exists $a$ such that in $G_{j+a}$, the endpoints of the edge colored $10$ at the center of $S$
are adjacent to a path containing $3$ different colors. These $3$ colors are distinct from $0$ and $1$.
So without loss of generality, $G_{j+a}$ contains the three words of the form $F_x=0\cdots1x0\cdots1$ with $x\in\acc{2,3,4}$.

By Lemma~\ref{outer} and Lemma~\ref{lem:outer}, there exists $b$ such that in $G_{j+a+b}$, the vertex colored~$1$ that is the common suffix of the words $F_x$
is adjacent to an outerplanar containing $6$ different colors. Let $Q_6\subset A_{10}$ be the set of these $6$ colors.
First, $1\not\in Q_6$ because the coloring is proper. If there exists $x\in Q_6\cap\acc{2,3,4}$, then $G_{j+a+b}$ would contain the square $F_xx=0\cdots1x0\cdots1x$
with period $q+1$. That would contradict the minimality of $q$, so $Q_6\cap\acc{2,3,4}=\emptyset$.
Therefore, $Q_6=\acc{0,5,6,7,8,9}$.
Since $0\in Q_6$, $G_{j+a+b}$ contains the overlap $0\cdots10\cdots10$.

We have seen that the square with period $q$ and length $2q$ in $G_{j+a}$ extends to one letter to right in $G_{j+a+b}$ to give the overlap with period $q$ and length $2q+1$.
Similarly, the suffix of length $2q$ of this overlap in $G_{j+a+b}$ is again a square that extends to an overlap in $G_{j+2(a+b)}$.
So $G_{j+2(a+b)}$ contains a repetition with period $q$ and length $2q+2$. By induction, $G_{j+2q(a+b)}$ contains a repetition with period $q$ and length $4q$.
This repetition in $G_{j+2q(a+b)}$ is a fourth power with period $q$ and thus a square with period $2q$.
Finally, this contradicts the minimality of $q$. 


\end{document}